\documentclass[12pt]{article}

\usepackage{amsmath, amsthm, amssymb, fullpage}
\usepackage{multirow}
\usepackage{diagrams}

\def\N{\mathbb{N}}
\def\Z{\mathbb{Z}}
\def\Zp{\Z_p}
\def\Zn{\Z_n}
\def\R{\mathbb{R}}
\def\RP{\R P}
\def\half{\frac{1}{2}}
\def\Ztwo{\Z_2}
\newcommand{\toby}[1]{\stackrel{#1}{\longrightarrow}}
\def\Lg{L\mathfrak{g}}
\def\Lghat{\widehat{\Lg}}
\def\LGhat{\widehat{LG}}
\def\Hol{\mathrm{Hol}}
\def\pr{\mathrm{pr}}

\newtheorem{thm}{Theorem}[section]
\newtheorem{lemma}[thm]{Lemma}
\newtheorem{remark}[thm]{Remark}
\newtheorem{defn}[thm]{Definition}
\newtheorem{fact}[thm]{Fact}
\newtheorem{prop}[thm]{Proposition}

\title{Pre-quantization of the Moduli Space of Flat $G$-Bundles over a Surface}

\author{Derek Krepski}

\begin{document}

\maketitle


\begin{abstract}
For a simply connected, compact, simple Lie group $G$, the moduli space of flat $G$-bundles over a closed surface $\Sigma$ is known to be pre-quantizable at integer levels.  For non-simply connected $G$, however, integrality of the level is not sufficient for pre-quantization, and this paper determines the obstruction --- namely a certain cohomology class in $H^3(G^2;\Z)$ --- that places further restrictions on the underlying level. The levels that admit a pre-quantization of the moduli space are determined explicitly for all non-simply connected, compact, simple Lie groups $G$.

 \end{abstract}

\section{Introduction}

Let $G$ be a compact, connected simple Lie group (not necessarily simply connected) with universal cover $\tilde{G}$, and let $\Sigma$ be an oriented compact surface of genus $g$ with one boundary component. The moduli space $M(\Sigma)\cong \mathrm{Hom}(\pi_1(\Sigma),G)$ of framed flat $G$-bundles over $\Sigma$ is  a smooth finite dimensional manifold.  $M(\Sigma)$ is naturally a quasi-Hamiltonian $\tilde{G}$-space \cite{AMM}, and this paper considers the pre-quantization of $M(\Sigma)$ (see Section \ref{sec:quasi}) .  To be specific, the definition of a quasi-Hamiltonian $\tilde{G}$-space involves a choice of a level $l>0$ --- a multiple of the basic inner product on the Lie algebra of $\tilde{G}$ --- and this work determines the levels $l$ for which the moduli space $M(\Sigma)$ is pre-quantizable.  The main result of this paper is that $M(\Sigma)$ is pre-quantizable if and only if the underlying level $l$ is an integer multiple of the integer $l_0$ listed in Table \ref{summary}. (When $G$ is simply connected, it is well known that $l_0=1$; therefore,  this paper is primarily concerned with $G$'s that are not simply connected.)

\begin{table}[ht]
\centering
\caption{The integer $l_0$ for non-simply connected $G$}
\begin{tabular}{|c|c|c|c|c|}
\hline

 $\tilde{G}$ & $Z(\tilde{G})$ & $\Gamma$ & $G$ &  $l_0$  \\
  \hline\hline
  \multirow{2}{*}{$SU(n)$, $n\geq 2$}   & \multirow{2}{*}{$\Zn$}  & $\Zn$ & $PSU(n)$  & $n$ \\
  		   &  &$\Z_l$, $1<l<n$& $SU(n)/\Z_l$  &  $\mathrm{ord}_l(\frac{n}{l})$\\
  \hline
  $Sp(n)$, $n\geq 1$   &$ \Ztwo$  & $\Ztwo$ & $PSp(n)$  & $1+\rho(n)$  \\
  \hline
    $Spin(n)$, $n\geq 7$ odd   & $\Ztwo$ &  $\Ztwo$ & $SO(n)$  &  $1$ \\
  \hline
  $Spin(2n)$, & \multirow{2}{*}{$\Z_4$} & $\Z_4$ & $PSO(2n)$  & $4$ \\
	$n\geq 5$	 odd &	& $\Ztwo$ & $SO(2n)$   &  $1$ \\
\hline
 \multirow{3}{*}{$Spin(4n)$, $n\geq 2$}  & \multirow{3}{*}{$\Ztwo \times \Ztwo$} & $ \Ztwo \times \Ztwo$ & $PSO(4n) $   & $2$ \\
  	  & &  $\Ztwo$ & $Ss(4n)$  &  $1 + \rho(n)$ \\
	  & & $\Ztwo$ & $SO(4n)$ & $1$\\
\hline
  $E_6$  & $\Z_3$ & $\Z_3$ &$PE_6$  & $3 $\\
  $E_7$  &  $\Ztwo$& $\Ztwo$ & $PE_7$   & $2 $\\
  \hline
\end{tabular}
\label{summary}
\end{table}

In Table 1,   $\Gamma=\pi_1(G)$, and hence $G=\tilde{G}/\Gamma$. Also,  $\rho:\Z \to \Ztwo$ is reduction mod $2$, and $\mathrm{ord}_l(x)$ denotes the order of $x$ mod $l$ in $\Z_l$ . (See  Section \ref{section:calculation} with regards to other notational ambiguities.)

A curiosity is that a similar table (Proposition 3.6.2 in \cite{TL}) appears in a paper by Toledano-Laredo in a different context. In \cite{TL}, Toledano-Laredo classifies irreducible positive energy representations of loop groups $LG$. The reason for the similarity in the tables is not yet understood.

The results in this paper extend  readily to the moduli space $M(\Sigma, \mathcal{C})$, the moduli space of flat $G$-bundles over $\Sigma$ with holonomy around $\partial\Sigma$ in the prescribed conjugacy class $\mathcal{C} \subset G$. That is, since $M(\Sigma,\mathcal{C})$ is the symplectic quotient of the fusion product $M(\Sigma) \circledast \mathcal{C}$, Propositions   \ref{prop:qeqpreq=reducedpreq} and \ref{prop:fusion} (which describe how pre-quantization interacts with   symplectic reduction and fusion) completely address the pre-quantization of $M(\Sigma,\mathcal{C})$ (see Remarks \ref{remark:reduction} and \ref{conjugacy}).

The plan of attack is to rephrase the problem of determining $l_0$ in terms of the cohomology of Lie groups.  Theorem \ref{thm:coho-calc-equiv} realizes the integer $l_0$ as the generator of $\ker \tilde\phi^* \subset \Z\cong H^3(\tilde{G};\Z)$, where  $\tilde\phi:G^2\to \tilde{G}$ is the canonical lift of the commutator map $\phi:G^2\to G$.  Using the classification of compact simple Lie groups, and known results about their cohomology rings (as Hopf algebras), the integer $l_0$ is then computed in Section \ref{section:calculation}.

It is not surprising that the integer $l_0$ is the result of a cohomology calculation, since this paper views pre-quantization in the framework of cohomology to begin with. As discussed in Section \ref{sec:quasi}, a pre-quantization of a quasi-Hamiltonian $G$-space is an \emph{integral lift} of a certain $G$-equivariant de Rham cohomology class that is associated to the quasi-Hamiltonian $G$-space (see Remark \ref{remark:pair}). This is analogous to the pre-quantization of an ordinary symplectic manifold $(M,\omega)$, where  a pre-quantization is usually defined as a complex line bundle over $M$, equipped with a connection whose curvature is the corresponding symplectic form $\omega$ (see \cite{GGK}).  By the well known correspondence between line bundles over $M$ and $H^2(M;\Z)$, a pre-quantization gives an integral lift of the de Rham cohomology class $[\omega]\in H^2_{dR}(M)$ determined by the symplectic form.   One may make similar analogies with the equivariant pre-quantization of Hamiltonian $G$-manifolds, $G$-equivariant cohomology classes, and $G$-equivariant line bundles as well; the interested reader is referred to \cite{GGK}.

As discussed in \cite{AMM}, the theory of quasi-Hamiltonian $G$-spaces is equivalent to the theory of Hamiltonian loop group spaces; therefore, pre-quantization may be interpreted in either setting. This  is considered in  Appendix \ref{sec:loopgroup}.
Pre-quantization may also be studied in the language of pre-quasi-symplectic groupoids --- see \cite{LX} for such a treatment. Also of interest is the related work \cite{FGK}.

\noindent\emph{Acknowledgements:} I would like to thank my advisors Eckhard Meinrenken and Paul Selick for their guidance, support, and patience. My sincere thanks to Lisa Jeffrey as well, for many helpful discussions.

\section{Preliminaries} \label{sec:quasi-intro}

This section is a review  of quasi-Hamiltonian $G$-spaces, and serves mainly to establish notation.  The reader should consult \cite{AMM} for details.

For a (not necessarily simple) compact Lie group $G$, choose an
invariant positive definite inner product $(\cdot,\cdot)$ on its
Lie algebra $\mathfrak{g}$, and let
$\eta=\frac{1}{12}(\theta^L,[\theta^L,\theta^L])$
$=\frac{1}{12}(\theta^R,[\theta^R,\theta^R]) \in \Omega^3(G)$ be
the canonical 3-form for this choice of inner product. Here
$\theta^L$ and $\theta^R$ are the left and right invariant
Maurer-Cartan forms, respectively. The group $G$ itself is considered as a
$G$-manifold, acting by conjugation.

\begin{defn}  \label{defn:qhspace} A quasi-Hamiltonian $G$-space
is a triple $(M,\omega,\phi)$ consisting of a $G$-manifold $M$,
an invariant 2-form $\omega$ on $M$, and an equivariant map
(called the moment map) $\phi:M\to G$ satisfying:
\begin{enumerate}
\item[(1)] $d\omega +\phi^*\eta =0$

\item[(2)] $\iota_{\xi^\sharp}\omega = -\half
\phi^*(\theta^L+\theta^R,\xi)$ for all $\xi\in\mathfrak{g}$

\item[(3)] At every point $p\in M$,
$\ker\omega_p \cap \ker{\mathrm{d}\phi|_p} = \{0\}$

\end{enumerate}
\end{defn}

\begin{remark} \label{remark:pair} Notice that condition (1) of Definition \ref{defn:qhspace} says
that the pair $(\omega,\eta)$ defines a cocycle of dimension $3$
in $\Omega^*(\phi)$, the algebraic mapping cone of
$\phi^*:\Omega^*(G)\to\Omega^*(M)$. Thus the pair determines a
cohomology class $[(\omega,\eta)]\in H^3(\phi\, ;\R)$.
In fact, conditions (1) and (2) above can be re-expressed in
terms of the Cartan model for equivariant differential forms on
$M$ (see \cite{GS}) . Specifically, (1) and (2) may be replaced by the single
relation
\[d_G\omega+\phi^*\eta_G =0 \] where $\omega$ is viewed as an
equivariant differential form and $\eta_G$ is the equivariant
differential form given by
$\eta_G(\xi)=\eta+\half(\theta^L+\theta^R,\xi)$ in the Cartan
model. Therefore conditions (1) and (2)  give an
equivariant cohomology class $[(\omega,\eta_G)]\in H^3_G(\phi\,
;\R)$.  This is the salient feature of the above definition, for the purposes of this paper. (In fact, this paper does not make use of condition (3) above.)
\end{remark}

\noindent \emph{N.B.} Throughout this paper, de Rham cohomology $H^*_{dR}(\quad)$ will be identified with singular cohomology with real coefficients $H^*(\quad;\R)$ via the canonical isomorphism.

 If $G$ is simple, the  invariant
inner product $(\cdot,\cdot)$ on $\mathfrak{g}$ is necessarily a multiple of the basic inner product $B$: the invariant inner product normalized to make short co-roots have length $\sqrt{2}$.  The canonical $3$-form $\eta$ associated to the basic inner product $B$ will be denoted $\eta_1$.

\begin{defn} \label{defn:qHlevels}
Let $G$ be a simple compact Lie group. The quasi-Hamiltonian $G$-space $(M,\omega,\phi)$ is said to be at \emph{level} $l>0$ if the invariant inner product chosen on $\mathfrak{g}$ is $l B$.
\end{defn}

The following two facts are included for reference.  Fact \ref{fact:fusion} describes how  the product of two quasi-Hamiltonian $G$-spaces is naturally a quasi-Hamiltonian $G$-space, and Fact \ref{fact:reduction} is the quasi-Hamiltonian analogue of Meyer-Marsden-Weinstein reduction in symplectic geometry.

\begin{fact}[Fusion] \label{fact:fusion} Let $(M_1,\omega_1,\phi_1)$ and $(M_2,\omega_2,\phi_2)$ be quasi-Hamiltonian $G$-spaces. Then $M_1\times M_2$ is a quasi-Hamiltonian $G$-space, with diagonal $G$-action, invariant $2$-form $\omega=\pr_1^*\omega_1 + \pr_2^*\omega_2 +
\half(\pr_1^*\phi_1^*\theta^L,\pr_2^*\phi_2^*\theta^R)$, and group-valued moment map $\phi=\mu\circ(\phi_1\times\phi_2)$, where $\mu:G^2\to G$ is group multiplication. This quasi-Hamiltonian $G$-space is called the \emph{fusion product}, and is denoted $(M_1\circledast M_2,\omega,\phi)$.
\end{fact}

\begin{fact}[Reduction] \label{fact:reduction} Let $(M,\omega,\phi)$ be a quasi-Hamiltonian $G$-space.
For any $h\in G$ that is a regular value, the centralizer $Z_h$ of $h$ acts locally freely on the level set $\phi^{-1}(h)$, and the restriction of $\omega$ to $\phi^{-1}(h)$ descends to a symplectic form on the orbifold $\phi^{-1}(h)/Z_h$. In particular, if the identity $e$ is a
regular value, and $j:\phi^{-1}(e)\to M$ denotes the inclusion, then $j^*\omega$ descends to a symplectic form
$\omega_{red}$ on $M/\!/G:=\phi^{-1}(e)/G$.
\end{fact}

In analogy with co-adjoint orbits $\mathcal{O}\subset \mathfrak{g}^*$ ---  central examples in the theory of  Hamiltonian $G$-spaces ---  conjugacy classes $\mathcal{C}\subset G$ are  important examples of  quasi-Hamiltonian $G$-spaces.  The group-valued moment map is just the inclusion $\iota:\mathcal{C}\hookrightarrow G$.
Notice that $h$ is a regular value for the moment map $\phi:M\to G$ if and only if the identity $e\in G$ is a regular value for the moment map of the fusion product $\Phi=\mu\circ (\phi\times \iota):M\times \mathcal{C} \to G$, where $\mathcal{C}\subset G$ is the conjugacy class containing $h^{-1}$.  In this case, there is a canonical identification of quotients:
\[\phi^{-1}(h)/Z_h \cong  \Phi^{-1}(e)/G=(M\circledast \mathcal{C})/\!/G. \]

Finally, recall another main example of a quasi-Hamiltonian $G$-space: the double $\mathbf{D}(G):=G^2$ of a compact Lie group $G$.  In particular, there is a $G$-invariant 2-form $\omega$ such that $(G^2, \omega, \phi)$ is a quasi-Hamiltonian $G$-space with moment map $\phi(a,b)=aba^{-1}b^{-1}$.  Here $G$ acts by conjugation on each factor. Notice that  the moment map $\phi$ lifts canonically to $\tilde{G}$, the universal cover of $G$, and $G^2$ may be viewed (as it is in this work) as a quasi-Hamiltonian $\tilde{G}$-space.

$G^{2g}$, the $g$-fold fusion product of $\mathbf{D}(G)$, has a geometric interpretation that is discussed next for compact, connected, simple Lie groups $G$.  Let $\Sigma$ be a compact oriented surface of genus $g$ with one boundary component, and fix a basepoint $b\in \partial\Sigma$. 
The moduli space $M(\Sigma)$ is the set of equivalence classes of triples $(P,\theta,\psi)$, where $P \to \Sigma$ is a principal $G$-bundle equipped with a flat connection $\theta$ and $\psi: P|_b\cong G$ is a \emph{framing}.   Two such triples $(P_i,\theta_i, \psi_i)$ are equivalent if there is a gauge transformation $f:P_1\to P_2$ with $f^*\theta_2=\theta_1$, and $f^*\psi_2=\psi_1$.  Since $\Sigma$ is homotopy equivalent to a wedge of circles, every principal $G$-bundle over $\Sigma$ is trivial. Therefore $M(\Sigma)$ may also be viewed as the space of flat connections on $\Sigma\times G\to \Sigma$, up to based gauge transformations (i.e. gauge transformations $f$ whose restriction to $\{b\} \times G$ is the identity).  Holonomy defines  the natural identification $M(\Sigma)\cong \mathrm{Hom}(\pi_1(\Sigma),G)$, which may further be identified with $G^{2g}$ after choosing  generators of  $\pi_1(\Sigma)$.

The moment map $\phi_g:G^{2g}\to G$, $(a_1, b_1, \ldots, a_g, b_g)\mapsto \prod[a_i,b_i]$, may then be interpreted as sending an equivalence class of flat connections $[\theta]\in M(\Sigma)$ to its holonomy around the boundary $\partial\Sigma$. The level set $\phi_g^{-1}(e)$ is then simply the equivalence class of flat connections (up to based gauge transformations) whose holonomy around $\partial\Sigma$ is trivial.  Now, $M(\Sigma)$ is equipped with a $G\cong \mathcal{G}(\Sigma)/\mathcal{G}_b(\Sigma)$ action, where $\mathcal{G}_b(\Sigma)$, and $\mathcal{G}(\Sigma)$ denote the based and full gauge group, respectively.  Therefore, the symplectic quotient $\phi_g^{-1}(e)/G$ may be viewed as the moduli space of  flat connections on $\Sigma\times G \to \Sigma$ (up to gauge transformations) with trivial holonomy around $\partial\Sigma$.

Recall that the kernel $\Gamma$ of the universal covering homomorphism $\pi:\tilde{G}\to G$ is a finite central subgroup of $\tilde{G}$.  The canonical lift $\tilde\phi_g:G^{2g}\to \tilde{G}$ of the moment map $\phi_g$ may be interpreted as sending an equivalence class of connections $\theta$ to the holonomy of the $\tilde{G}$-connection $(1\times \pi)^*\theta$ around $\partial\Sigma$ for the $\tilde{G}$-bundle $\Sigma \times \tilde{G} \to \Sigma$.  It is easy to see that 
\begin{equation} \label{decomp}
\phi_g^{-1}(e)/G=\coprod_{z\in \Gamma} \tilde\phi_g^{-1}(z)/\tilde{G}
\end{equation}
which describes the connected components of  $M(\Sigma)/\!/G$.  (Note that the fibers $\tilde\phi_g^{-1}(z)$ are connected by Theorem 7.2 in \cite{AMM}.)

The decomposition (\ref{decomp}) is illuminating in the setting where the underlying surface has no boundary.
Indeed, let $\hat{\Sigma}$ be the surface obtained by attaching a disc $D$ to $\Sigma$ with attaching map $\partial D \toby{=}\partial \Sigma \subset \Sigma$.  Recall that there is a bijective correspondence between principal $G$-bundles over $\hat\Sigma$ and $\pi_1(G)\cong \Gamma$: every principal $G$-bundle over $\hat\Sigma$ can be constructed by gluing together trivial bundles over $\Sigma$ and $D$ with a transition function $S^1\!\!=\!\!\Sigma\cap D \to G$.
The moduli space of flat $G$-bundles over $\hat\Sigma$ (up to gauge transformation), denoted $M(\hat\Sigma)$, can thus be identified with the symplectic quotient $M(\Sigma)/\!/G$, and the above decomposition describes the components of $M(\hat\Sigma)$ in terms of the bundle types enumerated by $\Gamma$. That is,  $\tilde\phi_g^{-1}(z)/\tilde{G}$ may be thought of as the moduli space of flat connections on $P \to \hat\Sigma$ (up to gauge transformation) where $P \to \hat\Sigma$ is a principal $G$-bundle corresponding to $z\in \Gamma$.

Similarly, the symplectic quotients $\phi_g^{-1}(h)/Z_h\cong (M(\Sigma)\circledast \mathcal{C})/\!/G$ at other regular values $h$ may be viewed as the moduli space of flat connections whose holonomy around $\partial\Sigma$ lies in the conjugacy class $\mathcal{C}$ containing $h^{-1}$.  And similar to (\ref{decomp}), there is the decomposition
\begin{equation} \label{decomp2}
 (M(\Sigma)\circledast \mathcal{C})/\!/G = \coprod_j (M(\Sigma)\times \tilde{\mathcal{C}}_j)/\tilde{G}
\end{equation}
where $\tilde{\mathcal{C}}_j\subset \tilde{G}$ are the conjugacy classes that cover $\mathcal{C}\subset G$.


\section{Pre-quantization of Quasi-Hamiltonian $G$-spaces} \label{sec:quasi}

This section addresses pre-quantization in the context of quasi-Hamiltonian $G$-spaces. As mentioned in the introduction, one may wish to compare this with the situation for ordinary symplectic manifolds, and Hamiltonian $G$-manifolds (see \cite{GGK}).

\begin{defn} \label{preq}
A pre-quantization of a quasi-Hamiltonian $G$-space $(M,\omega,\phi)$ is an integral lift of $[(\omega,\eta_G)]\in H^3_G(\phi;\R)$.  That is, a pre-quantization is a cohomology class $\alpha\in H^3_G(\phi;\Z)$ satisfying $\iota_\R(\alpha)=[(\omega,\eta_G)]$, where $\iota_\R:H^*_G(\quad;\Z)\to H^*_G(\quad;\R)$ is the coefficient homomorphism.
\end{defn}

\begin{remark} There is a geometric interpretation of Definition \ref{preq}, in terms of \emph{relative gerbes} and \emph{quasi-line bundles} (see \cite{Sh}) that is analogous to the situation for symplectic manifolds, where pre-quantization is defined as a line bundle \cite{GGK}.
\end{remark}

Proposition \ref{prop:eq-pqispq} below shows that a pre-quantization may  be viewed as an ordinary cohomology class when $G$ is  simply connected.

\begin{lemma} \label{lemma:simplyG1} Let $G$ be  simply connected, and let $X$ be
some $G$-space. Then (with any coefficients) $H^i_G(X)\cong
H^i(X)$ for $i=1,2$, and the canonical map $H^3_G(X)\to H^3(X)$ is
injective.
\end{lemma}
\begin{proof} Consider the Serre spectral sequence for
the fibration $X\toby{\pi} X_G\to BG$ where $X_G$ denotes the
Borel construction. As $G$ is simply connected, $H^i(BG)=0$ for
$i=1,2,3$. Therefore the first non-zero differential (i.e. the
transgression) is $d:H^3(X)\to H^4(BG)$, and therefore
$H^i(X)=H^i(X_G)$ for $i=1,2$, and $\pi^*:H^3(X_G)\to H^3(X)$ is
injective. \end{proof}

\begin{lemma} \label{lemma:simplyG2} If $G$ is simply
connected, then (with any coefficients) $H_G^3(G)\cong H^3(G)$.
\end{lemma}
\begin{proof}  From the long exact sequence of homotopy groups for the fibration  $G \times EG \toby{p} G_G \to BG$, it is clear that $p$ (the quotient map in the Borel construction) is an isomorphism on $\pi_k(\quad)$ for $k=1$ and $2$, and a surjection for $k=3$.  In particular $\pi_1(G_G)=\pi_2(G_G)=0$. Now, the connecting homomorphism $\partial:\pi_4(BG)\to\pi_3(G)$ may be viewed as  the induced homomorphism on  $\pi_3(\quad)$ of the map $j:G\to G\times EG$, which includes $G$ as the fiber of $p$. But $j$ is null homotopic, as it sends $g\in G$ to $g\cdot(e,*)=(e,g\cdot *)$, thus factoring through the contractible space $EG$. Therefore $j$ induces the zero map on $\pi_3(\quad)$ (i.e. $\partial=0$), and hence $p$ induces an isomorphism on $\pi_3(\quad)$ as well. Therefore by the Hurewicz Theorem, $p$ induces an isomorphism on $H_3(\quad;\Z)$, and therefore on $H^3(\quad;\Z)$ as well.
 \end{proof}

\begin{prop} \label{prop:eq-pqispq} Suppose $G$ is simply connected, and
$(M,\omega,\phi)$ is a quasi-Hamiltonian $G$-space. Then
$(M,\omega,\phi)$ admits a pre-quantization if and only if
the cohomology class $[(\omega,\eta)]\in H^3(\phi;\R)$ is integral (i.e. in the image of the coefficient homomorphism $\iota_\R:H^3(\phi;\Z) \to H^3(\phi;\R)$).
\end{prop}
\begin{proof} Easy applications of Lemma
\ref{lemma:simplyG1} with $X=M$ and $X=G$, Lemma
\ref{lemma:simplyG2}, and the five-lemma show that
$H^3_G(\phi)\cong H^3(\phi)$. \end{proof}

Under certain circumstances, the symplectic quotient of a quasi-Hamiltonian $G$-space will be a smooth symplectic manifold, and one may ask whether a pre-quantization \emph{descends} to a pre-quantization  of the symplectic quotient. (Recall that a pre-quantization of a symplectic manifold $(M,\omega)$ may be defined as a cohomology class $\alpha\in H^2(M;\Z)$ with $\iota_\R(\alpha)=[\omega]$ --- see \cite{GGK}.)

\begin{prop} \label{prop:qeqpreq=reducedpreq} Let $(M,\omega,\phi)$ be a quasi-Hamiltonian
$G$-space, and suppose the identity $e\in G$ is a regular value for the moment map $\phi$.  If $(M,\omega,\phi)$ admits a pre-quantization, then the cohomology class $[j^*\omega] \in H^2_G(\phi^{-1}(e);\R)$ is integral.
\end{prop}
\begin{proof} There is a canonical map $H^3_G(\phi) \to H^2_G(\phi^{-1}(e))$ (with any coefficients) given by the composition of the map induced by restriction $H_G^3(\phi) \to H_G^3(\phi|_{\phi^{-1}(e)})$ and the projection onto the first summand $$H_G^3(\phi|_{\phi^{-1}(e)}) \cong H^2_G(\phi^{-1}(e))\oplus H_G^3(pt) \to H^2_G(\phi^{-1}(e)).$$ In other words, there is a diagram:
\[
\begin{array}{ccccc}
H^2_G(M) 	& \longrightarrow & H_G^3(\phi) 	& \longrightarrow & H_G^3(G) \\
	\downarrow  j^*		&	   &\downarrow	 	&	&\downarrow \\
	H^2_G(\phi^{-1}(e))	& \to& H_G^3(\phi|_{\phi^{-1}(e)})& \to & H^3_G(pt)\\
\end{array} \]

\noindent in which the bottom row is (canonically) split exact with any coefficients. Since $H^3_G(pt;\R)=0$ it suffices to check that (with real coefficients) the middle map sends the relative cohomology class $[(\omega,\eta_G)]$ to $[(j^*\omega,0)]$, which is clear.
\end{proof}

\begin{remark} \label{remark:reduction} If in the previous proposition $G$ acts freely on the level set $\phi^{-1}(e)$, one finds that the symplectic quotient $M/\!/G=\phi^{-1}(e)/G$ is prequantizable, since $H^2_G(\phi^{-1}(e)) \cong H^2(M/\!/G)$, and therefore the cohomology class $[\omega_{red}]$ is integral.  Said differently, the previous proposition
says that if $(M,\omega,\phi)$ is a pre-quantized quasi-Hamiltonian $G$-space  and $e\in G$ is a regular value, then  there is  a $G$-equivariant pre-quantum line bundle over the level set $\phi^{-1}(e)$ (see \cite{GGK}). And under the additional hypothesis that $G$ acts freely on $\phi^{-1}(e)$,  this $G$-equivariant line bundle descends to a pre-quantum line bundle over the symplectic quotient.  Of course, if $G$ only acts locally-freely on $\phi^{-1}(e)$, one obtains instead an \emph{orbi-bundle} over the symplectic quotient. (The reader may wish to consult \cite{MS} for a more thorough account of pre-quantization of singular spaces.)
 \end{remark}

Proposition \ref{prop:fusion}  shows that the fusion product $M_1\circledast M_2$ of two pre-quantizable quasi-Hamiltonian $G$-spaces is pre-quantizable.  In fact, the proof of the proposition shows that the pre-quantization of $M_1\circledast M_2$ is canonically obtained from the pre-quantizations of $M_1$ and $M_2$.

\begin{prop} \label{prop:fusion} Let $G$ be simply connected. If $(M_1,\omega_1,\phi_1)$ and $(M_2,\omega_2,\phi_2)$
are pre-quantized quasi-Hamiltonian $G$-spaces, then their fusion product $(M_1 \circledast M_2,\omega,\phi)$ inherits a pre-quantization.
\end{prop}


\begin{proof} 
First, recall that for a manifold $X$, the sub-complex of smooth singular cochains $S_{sm}^*(X;\Z)$ is chain homotopy equivalent to the singular cochain complex $S^*(X;\Z)$. Both $S^*_{sm}(X;\Z)$ and the de Rham complex of differential forms on $X$ may be viewed  as  sub-complexes of $S^*_{sm}(X;\R)$ \cite{Ma}.

 By Proposition \ref{prop:eq-pqispq} it suffices to construct a cohomology class $[\alpha]\in S^3_{sm}(\phi;\Z)$ such that $\iota_\R([\alpha])=[(\omega,\eta)]$, where  $\omega=\pr_1^*\omega_1 + \pr_2^*\omega_2 + \half(\pr_1^*\phi_1^*\theta^L,\pr_2^*\phi_2^*\theta^R)$, and $\iota_\R$ is the coefficient homomorphism.

 To begin, choose a cochain representative $\bar{\eta} \in S^3_{sm}(G;\Z)$ (unique up to coboundary) that satisfies $\iota_\R([\bar{\eta}])=[\eta]$. For concreteness, let $\varepsilon \in S^2_{sm}(G;\R)$ be a smooth cochain that satisfies $\bar{\eta} -\eta = d\varepsilon$.
Since the cohomology class $[\bar{\eta}]\in H^3(G;\Z)$ is primitive (as $G$ is simply connected), there exists a smooth cochain $\tau \in S^2_{sm}(G\times G;\Z)$ (unique up to coboundary, given the choice of $\bar\eta$) that satisfies $d\tau = \mu^*\bar{\eta} - \pr_1^*\bar{\eta} - \pr_2^*\bar{\eta}$.

Since (for $i=1$ and $2$) $(M_i,\omega_i,\phi_i)$  is pre-quantized, there exists a
cochain representative $\bar\omega_i \in S^2_{sm}(M_i;\Z)$ (unique up to coboundary, given the choice of $\bar\eta$) such that $\bar\omega_i -\omega_i + \phi_i^*\varepsilon$ is exact in $S^2_{sm}(M_i;\R)$, so that $[(\bar\omega_i, \bar\eta)]$ is the given pre-quantization of $(M_i,\omega_i, \phi_i)$. 

Define the smooth relative cochain
 $\alpha=(\pr_1^*\bar\omega_1 + \pr_2^*\bar\omega_2 - (\phi_1\times \phi_2)^*\tau,\bar{\eta})$ in $S^3_{sm}(\phi;\Z)$.   It must be verified that (a) $\alpha$ is  a relative cocycle that is cohomologous to $(\omega,\eta)$ with real coefficients, and (b) $[\alpha]$ is independent of the choice of $\bar\eta$, and the subsequent choices of $\tau$ and $\bar\omega_i$. It will follow that $[\alpha]$  is the desired pre-quantization. 

It is clear that $d\alpha = 0$ since
\begin{align*}
d(\pr_1^*\bar\omega_1 + \pr_2\bar\omega_2 -(\phi_1\times\phi_2)^*\tau)
	& =\pr_1^*(-\phi_1^*\bar{\eta}) + \pr_2^*(-\phi_2^*\bar{\eta}) \\
	& \quad -(\phi_1\times\phi_2)^*(\mu^*\bar{\eta} - \pr_1^*\bar{\eta} - \pr_2^*\bar{\eta}) \\
	&=\pr_1^*(-\phi_1^*\bar{\eta}) + \pr_2^*(-\phi_2^*\bar{\eta}) \\
	&\quad -\phi^*\bar{\eta} + \pr_1^*\phi_1^*\bar{\eta} + \pr_2^*\phi_2^*\bar{\eta} \\
	& = -\phi^*\bar{\eta}.
\end{align*}
\noindent And $\alpha$ is cohomologous to $(\omega,\eta)$ with real coefficients, since 
\begin{align*}
\alpha - (\omega,\eta) & = (\pr_1^*\bar\omega_1 + \pr_2^*\bar\omega_2 - (\phi_1\times \phi_2)^*\tau \\
 & \quad - (\pr_1^*\omega_1 + \pr_2^*\omega_2 + \half(\pr_1^*\phi_1^*\theta^L,\pr_2^*\phi_2^*\theta^R) ,\bar{\eta} -\eta) \\
& = (  \pr_1^*(\bar\omega_1-\omega_1) + \pr_2^*(\bar\omega_2 -\omega_2) - (\phi_1\times \phi_2)^*\tau \\
 & \quad -\half (\pr_1^* \phi_1^* \theta^L, \pr_2^*\phi_2^*\theta^R)    , d\varepsilon) \\
& = ( \pr_1^*(du_1 -\phi_1^*\varepsilon) + \pr_2^*(du_2 - \phi_2^*\varepsilon) \\
& \quad - (\phi_1\times \phi_2)^*(\tau + \half(\pr_1^*\theta^L, \pr_2^*\theta^R) ) , d\varepsilon) 
\end{align*}

\noindent for some primitives $u_i$ of $\bar\omega_i -\omega_i +\phi_i^*\varepsilon$. And since $\mu^*\eta=\pr_1^*\eta + \pr_2^* \eta -\half d(\pr_1^*\theta^L, \pr_2^*\theta^R)$, it follows that $d(\tau + \half (\pr_1^*\theta^L, \pr_2^*\theta^R)) = d(\mu^*\varepsilon - \pr_1^*\varepsilon - \pr_2^* \varepsilon)$. Therefore $\tau + \half  (\pr_1^*\theta^L, \pr_2^*\theta^R) -\mu^*\varepsilon + \pr_1^*\varepsilon + \pr_2^* \varepsilon$ is exact since $H^2(G\times G;\R)=0$. Continuing, 
\begin{align*}
\alpha - (\omega,\eta) & = ( \pr_1^*(du_1 -\phi_1^*\varepsilon) + \pr_2^*(du_2 - \phi_2^*\varepsilon) \\
& \quad - (\phi_1\times \phi_2)^*(\mu^*\varepsilon - \pr_1^*\varepsilon - \pr_2^* \varepsilon + dv) , d\varepsilon) \\
& = d(\pr_1^*u_1 + \pr_2^*u_2 - (\phi_1\times \phi_2)^*v, -\varepsilon)
\end{align*} 
\noindent where $v$ is some primitive of  $\tau + \half (\pr_1^*\theta^L,\pr_2^*\theta^R)- \mu^*\varepsilon + \pr_1^*\varepsilon + \pr_2^*\varepsilon$.

Lastly, it is straightforward to check that the cohomology class $[\alpha]$  does not depend on the choices made. Given a choice of $\bar\eta$, changing  $\bar\omega_i$ to $\bar\omega_i + d\beta_i$ changes $\alpha$ by $d(\beta_i,0)$. Changing $\tau$ to $\tau +d\gamma$ changes $\alpha$ by $d(-(\phi_1\times \phi_2)^*\gamma,0)$. Finally, if $\bar\eta$ is changed to $\bar\eta'=\bar\eta + d\rho$, then one may choose cochains $\bar\omega_i'=\bar\omega_i-\phi_i^*\rho$ and $\tau'=\tau +\mu^*\rho - \pr_1^*\rho - \pr_2^*\rho$ that satisfy the appropriate properties, producing the relative cocycle $\alpha'$. It is easy to see  that $\alpha'-\alpha=d(0,-\rho)$.  \end{proof}

\section{Pre-quantization of $G^{2g}$: the level as an obstruction} \label{subsection:restate}

Let $G$ be a simple, compact, connected Lie group. This 
section finds the obstruction to the pre-quantization of
$G^{2g}$, namely a certain cohomology class in $H^3(G^{2g};\Z)$. Theorem \ref{thm:coho-calc-equiv} states that this obstruction vanishes at certain levels (i.e. for certain choices of inner product on $\mathfrak{g}$).

\begin{lemma}Let $(M,\omega,\phi)$ be a level $l$ quasi-Hamiltonian
$G$-space. If $G$ is simply connected, and $(M,\omega,\phi)$ admits a pre-quantization, then $l \in \N$.
\end{lemma}
\begin{proof} Indeed, in the long exact sequence:
\[\cdots \to H^2(M;\R)\to H^3(\phi\, ;\R) \to H^3(G;\R)\to\cdots \]
the class $[(\omega,\eta)]$ maps to $[\eta]$. Therefore an
integral pre-image of $[(\omega,\eta)]$ gives an integral
pre-image of $[\eta]$.
It is well known that the generator of $H^3(G;\Z)$ maps to
$[\eta_1]$ under the coefficient homomorphism $\iota_\R$. (See
\cite{PS}, for example.) Since $\eta=l \cdot\eta_1$, $l \in \N$.
\end{proof}

\begin{prop} \label{prop:restateprob} Let $(M,\omega,\phi)$ be a
level $l$ quasi-Hamiltonian $G$-space. Assume $H^2(M;\R)=0$, and $G$
is simply connected.  Then $(M,\omega,\phi)$ admits a pre-quantization if and only if $l\in \mathrm{ker}\,\phi^*\subset
\Z=H^3(G;\Z)$.
\end{prop}
\begin{proof} The previous lemma shows that when $(M,\omega,\phi)$ admits a pre-quantization, $l$ is in the image of
$H^3(\phi\,;\Z)\to H^3(G;\Z)$ and therefore in the kernel of
$\phi^*$.

Conversely, suppose $l$ is in the kernel of $\phi^*$, or equivalently
in the image of $H^3(\phi\,;\Z)\toby{q}H^3(G;\Z)$. As in the
previous lemma, $l$ yields an integral lift of $[\eta]$. Choose
$\alpha\in H^3(\phi\,;\Z)$ with $q(\alpha)=l$, or with real coefficients,
$q(\iota_\R\alpha)=[\eta]$. Since $H^2(M;\R)=0$, the map $q$ is an injection
with real coefficients. Therefore $\iota_\R\alpha = [(\omega,\eta)]$.
 \end{proof}

 In particular, to decide which multiples $lB$ of the basic inner product $B$ admit a pre-quantization of $G^{2g}$, it suffices to determine which $l\in \N$ are in the kernel of $\tilde\phi_g^*$.  The kernel is a subgroup of $\Z$, and is therefore generated by some least positive integer $l_0$, which will be computed in Section \ref{section:calculation}.
 
 \begin{remark} \label{remark:rationalcoeffs}
 Note that when $G$ is simple, a quasi-Hamiltonian $G$-space necessarily satisfies $\phi^*=0$ in degree 3 cohomology with real coefficients (by condition (1) of Definition \ref{defn:qhspace}).  Therefore with integer coefficients, the image of $\phi^*$ is necessarily torsion.
 \end{remark}
 

The problem of determining
$l_0$ is independent of $g$. To see this,
let $\phi: G^2 \to G$ denote the commutator map, and
$\tilde\phi:G^2\to \tilde{G}$ its natural lift. Then  the moment
map $\phi_g:G^{2g}\to G$ lifts naturally as

\[ \tilde\phi_g=\mu_g\circ(
\underbrace{ \tilde\phi\times\cdots \times \tilde\phi}_g) \] where
$\mu_g:\tilde{G}^g\to\tilde{G}$ denotes successive multiplication
$\mu_g(x_1,\ldots,x_g)=x_1\cdots x_g$.

This gives $\mathrm{ker}\,\tilde\phi\subset
\mathrm{ker}\,\tilde\phi_g$. Indeed,the generator $z_3\in
H^3(\tilde{G};\Z)$ satisfies $\mu^*(z_3)=z_3\otimes 1 + 1 \otimes
z_3$. (Note that by the K\"unneth formula,
$H^3(\tilde{G}^2;\Z)=H^3(\tilde{G};\Z)\otimes H^0(\tilde{G} ;\Z)
\oplus H^0(\tilde{G};\Z)\otimes H^3(\tilde{G} ;\Z)$.) And
successive multiplication $\mu_g$ will therefore send $z_3$ to a
sum of tensors, where each tensor contains exactly one $z_3$ and
$(g-1)$ $1$'s. Therefore applying  $(\tilde\phi\times\cdots \times
\tilde\phi)^*$ to the resulting sum of tensors gives zero if
$\tilde\phi^*(z_3)=0.$

In fact, considering the inclusion $i:G^2\to G^{2g}$ as the first pair
of factors, one finds that $\tilde\phi$ factors as
\[ \tilde\phi=\tilde\phi_g\circ i \]
Therefore $\mathrm{ker}\,\tilde\phi_g \subset
\mathrm{ker}\,\tilde\phi$, which together with Proposition
\ref{prop:restateprob} shows:

\begin{thm}\label{thm:coho-calc-equiv} Let $G$ be a simple
compact connected Lie group, and $\tilde{G}$ its universal cover. Let $\phi$
denote the commutator map in $G$, and $\tilde\phi:G^2\to\tilde{G}$
its natural lift. The quasi-Hamiltonian $\tilde{G}$-space $(G^{2g},\omega,\tilde\phi_g)$ admits a pre-quantization if and only if the underlying level $l=ml_0$ for some $m\in \N$ (i.e. the chosen
invariant inner product on $\mathfrak{g}$ is
 $ml_0B$) where $l_0>0$ is the generator of $\mathrm{ker}\{
\tilde\phi^*:H^3(\tilde{G};\Z) \to H^3(G^2;\Z)\}\subset \Z$.
\end{thm}

\begin{remark} \label{conjugacy} As mentioned in the introduction, the above theorem combined with Theorem 6.1 of \cite{M}, and Proposition \ref{prop:fusion} give necessary and sufficient conditions for the pre-quantization of the quasi-Hamiltonian $\tilde{G}$-space $G^{2g}\times \tilde{\mathcal{C}}$, where $\tilde{\mathcal{C}}$ is a conjugacy class in $\tilde{G}$.  In particular, the decomposition  (\ref{decomp2}) in Section \ref{sec:quasi-intro} shows that these results extend to the symplectic quotients $(G^{2g}\times \mathcal{C})/\!/G$, where $\mathcal{C}$ is a conjugacy class in $G$.
\end{remark}


\section{Computing the integer $l_0$} \label{section:calculation}

The computation of $l_0$ (see Theorem \ref{thm:coho-calc-equiv}) uses the classification of compact simple Lie
groups, and known results about their cohomology. The strategy for
determining the image of $\tilde\phi^*$ with integer
coefficients is to compute $\tilde\phi^*$ with coefficients in $\mathbb{Q}$ (see Remark \ref{remark:rationalcoeffs}), and coefficients in
$\Zp$ for every prime $p$. This information is then pieced back together with the use of the Bockstein
spectral sequence in order to deduce the image of  $\tilde\phi^*$  with integer coeffiecients. (See \cite{S} for details about the Bockstein spectral sequence.)

Before beginning the calculation, it may be useful to establish some notation, and recall some elementary facts. 
Let $G$ be a topological group, with multiplication map $\mu_G:G\times G \to G$ (denoted $\mu$ when the group G is understood). The commutator map $\phi_G:G\times G\to G$ (denoted $\phi$ when the group $G$ is understood) is the composition given by:
\[G^2 \toby{\Delta\times\Delta} G^4 \toby{1\times T\times 1}
G^4 \toby{1\times 1\times c\times c} G^4 \toby{\mu\times \mu} G^2
\toby{\mu} G \] where:

\begin{itemize} \item $\Delta:G\to G^2$ denotes the
diagonal map $\Delta(g)=(g,g)$

\item $T:G^2\to G^2$ is the switching map $T(g,h)=(h,g)$

\item $c:G\to G$ is inversion $c(g)=g^{-1}$
\end{itemize}

These maps induce the following on cohomology with $\Zp$ coefficients
(using the identification $H^*(G)\otimes H^*(G)\cong
H^*(G\times G)$):
\begin{itemize}
\item $\Delta^*:H^*(G)\otimes H^*(G)\to H^*(G)$ is the algebra
multiplication

\item $T^*(u\otimes v)=(-1)^{|u||v|}v\otimes u$

\item $\mu^*:H^*(G)\to H^*(G)\otimes H^*(G)$ is the co-algebra
co-multiplication

\item $c^*(u)=-u$ if $u$ is primitive. If $u$ is not
primitive,  $c^*(u)$ may be calculated with the knowledge of $c^*(v)$
for all $v$ of smaller degree, and the fact that the composition $G\toby{\Delta} G^2\toby{1\times c} G^2\toby{\mu} G$ is trivial.
\end{itemize}

Throughout this section, $\pi: \tilde{G}\to G$ denotes the universal covering homomorphism, and hence $G=\tilde{G}/\Gamma$, where $\Gamma$ is a finite central subgroup of $\tilde{G}$. Let $\tilde\phi_G=\tilde\phi: G\times G \to \tilde{G}$ denote the canonical lift of $\phi_G$.

\subsection{The $A_n$'s}

Consider $\tilde{G}=SU(n)$, with center $\Zn$. Then the central
subgroups of $\tilde{G}$ are cyclic groups $\Z_l$ where $l$
divides $n$. Write $G=\tilde{G}/\Z_l$.

\begin{thm} \label{thm:An} Let $\tilde{G}=SU(n)$, and suppose $G=\tilde{G}/\Z_l$, where $l$ divides $n$. Then $l_0=\mathrm{ord}_l(\frac{n}{l})$, where $\mathrm{ord}_l(x)$ denotes the order of $x \mod l$ in $\Z_l$.
\end{thm}

\begin{proof}  It suffices to compute $\tilde\phi^*$ with coefficients in $\Zp$ for primes $p$
dividing $l$, since for primes $p$ not dividing $l$, the map
$\pi:\tilde{G} \to G$ induces an isomorphism on cohomology with
$\Zp$ coefficients. Therefore $\tilde\phi$ induces the
zero map on cohomology with coefficients in $\Zp$, as the Hopf
algebra $H^*(SU(n);\Zp)$ is co-commutative.

Recall the following computation due to Baum and Browder
\cite{BB}.

\begin{thm} \label{thm:BB}
Let $p$ be a prime dividing $l$.  Let $n=p^rn'$ and $l=p^sl'$
where $p$ is relatively prime to both $n'$ and $l'$. If $p\neq 2$
or $p=2$ and $s>1$, then there exist generators $x_{2i-1}\in
H^{2i-1}(G;\Zp)$, for $1\leq i\leq n$, $i\neq p^r$, and $y\in
H^2(G;\Zp)$ such that
\begin{itemize}
\item As an algebra,
\[H^*(G;\Zp)= \Lambda(x_1, \ldots, \hat{x}_{2p^r-1}, \ldots, x_{2n-1})
  \otimes \Zp[y]/(y^{p^r}) \]
where $\hat{}$ denotes omission.

\item The reduced coproduct is given by:
\[ \bar\mu^*(x_{2i-1})=\delta_{rs} x_1\otimes y^{i-1} +
     \sum_{j=2}^{i-1}\binom{i-1}{j-1} x_{2j-1}\otimes y^{i-j},\quad i\geq 2 \]
and $\bar\mu^*(x_1)=\bar\mu^*(y)=0$.
\end{itemize}
If $p=2$ and $s=1$ then the cohomology with $\Ztwo$ coefficients
is as above, with the additional relation that $y=x_1^2$.

\end{thm}

Theorem \ref{thm:An} is first proved for the case $G=\tilde{G}/\Zn$. This in turn will be considered in two cases, as it turns out $n=2$ mod $4$ requires special attention at the prime $2$. The case $G=\tilde{G}/\Z_l$, where $l<n$  will be considered later.

Suppose $l=n$. In this case, it will be shown
that $\tilde\phi^*(z_3)$ is the generator of the $\Zn$ summand in
$H^3(G\times G;\Z)$, and therefore $l_0=n$.

{\bf $n\neq 2\,\mathrm{mod}\,4$ }

In this case, one may deal with all primes $p$ dividing $n$
simultaneously.  By Corollary 4.2 in \cite{BB} $\pi^*(x_3)=z_3$ is the reduction $\mathrm{mod}\,p$ of the
generator with integer coefficients.  Since $\pi\tilde\phi=\phi$, it
suffices to calculate $\phi^*(x_3)$.  This is done next with $\Zp$ coefficients.

To compute $c^*(x_3)$,  notice that $x_3$ is not primitive. In
fact, according to Theorem \ref{thm:BB}, $\mu^*(x_3)=x_3\otimes 1
+ x_1\otimes y + 1\otimes x_3$. Therefore, $c^*(x_3)=-x_3+x_1y$.
Indeed,
\begin{align*}
0 & =  \Delta^*(1\times c)^*(x_3\otimes 1 + x_1\otimes y +
1\otimes x_3)
 \\
& =  \Delta^*(x_3\otimes 1 - x_1\otimes y + 1\otimes c^*(x_3))
\\
& =  x_3 - x_1 y +  c^*(x_3)
\end{align*}

To compute $\phi^*(x_3)$:

\begin{align*}
\phi^*(x_3) &= (\Delta\times\Delta)^*(1\times T\times
1)^*(1\times 1\times c\times c)^*(\mu\times \mu)^*\mu^*(x_3) \\
& = (\Delta\times\Delta)^*(1\times T\times 1)^*(1\times 1\times
c\times c)^*(x_3\otimes 1\otimes 1\otimes 1  \\
& \quad + x_1\otimes y \otimes 1\otimes 1 + 1\otimes x_3\otimes 1\otimes 1 + x_1\otimes 1\otimes
y\otimes 1 \\
& \quad + x_1\otimes 1\otimes 1\otimes y+1\otimes x_1 \otimes
y\otimes 1 + 1\otimes x_1 \otimes 1\otimes y \\
& \quad + 1\otimes 1\otimes x_3\otimes 1 + 1\otimes 1\otimes x_1\otimes y + 1\otimes 1\otimes1\otimes x_3)\\
&= (\Delta\times\Delta)^*(x_3\otimes 1\otimes 1\otimes 1  +
x_1\otimes 1 \otimes y\otimes 1 + 1\otimes 1\otimes x_3\otimes 1 + \\
&\quad - x_1\otimes y\otimes 1\otimes 1 - x_1\otimes 1\otimes
1\otimes y - 1\otimes y \otimes x_1\otimes 1 \\
& \quad - 1\otimes 1 \otimes x_1\otimes y - 1\otimes x_3\otimes 1\otimes1 + 1\otimes x_1y\otimes 1
\otimes 1 \\
& \quad + 1\otimes x_1\otimes 1\otimes y
- 1\otimes 1\otimes1\otimes x_3 +  1\otimes 1\otimes 1\otimes x_1y)\\
&=x_1\otimes y - y\otimes x_1
\end{align*}

By using the Bockstein spectral sequence, it will be shown next that
$x_1\otimes y - y\otimes x_1 \in H^3(G^2;\Zp)$ is the reduction
$\mathrm{mod}\,p$ of an integral class that generates a torsion
summand.

Write $n=p^rn'$ where $p$ does not divide $n'$.  As
$H_1(G;\Z)=\pi_1(G)=\Zn$, by the Universal Coefficient Theorem,
$H^2(G;\Z)$ contains a $\Zn\cong \Z_{p^r} \times \Z_{n'}$ summand.
This implies that $\beta^{(r)}(x_1)=y$ where $x_1$, and $y$ are
the cohomology classes in Theorem \ref{thm:BB}, and $\beta^{(r)}$
is the $r$-th Bockstein operator (see \cite{S}).

Therefore $\beta^{(r)}(x_1\otimes x_1)=y\otimes x_1 - x_1\otimes y $
in $H^*(G;\Zp)\otimes H^*(G;\Zp)\cong H^*(G\times G;\Zp)$.
This implies that $y\otimes x_1 - x_1\otimes y$ is the reduction
$\mathrm{mod}\,p$ of a generator of the $\Z_{p^r}$ summand in
$H^3(G\times G;\Z_{(p)})$ (i.e. ignoring torsion prime to $p$).

In other words, $\tilde{\phi}^*(z_3)=\phi^*(x_3)$ is the
reduction mod $p$ of a generator of the torsion summand in
$H^3(G\times G;\Z)$, which is $\Zn$. Thus the kernel of
$\tilde\phi^*$ is $n\Z$, and $l_0=n$.

{\bf
 $n=2\,\mathrm{mod}\,4$ }

First observe that the calculations
above remain valid for all primes $p\neq 2$ dividing $n$. That is,
for such primes, $\tilde\phi^*(z_3)=\phi^*(x_3)$ is the reduction
mod $p$ of a generator of the $\Zn$ summand in $H^3(G^2;\Z)$.
The same is true at the prime $2$, as shown next.

This is done by showing that $\tilde\phi_*$ is non-zero with $\Ztwo$
coefficients, and hence $\tilde\phi^*$ is non-zero as well with
$\Ztwo$ coefficients. Then, using the Bockstein spectral sequence, it will be verified that $\tilde\phi^*(z_3)$ is
indeed the reduction mod $2$ of the generator of the $\Zn$ summand
in $H^3(G^2;\Z)$.

In fact, it will suffice to consider the case $G=PSU(2)$, since there is a homomorphism $j:PSU(2)\to PSU(n)$ that induces isomorphisms on
$H^q(\quad;\Ztwo)$ when $q\leq 4$.
To see this, consider the diagonal inclusion $\iota: SU(2)\to SU(n)$ that sends a matrix $A$ to $A\oplus \cdots \oplus A$, which induces a homomorphism $j:PSU(2) \to PSU(n)$. Since $n=2\cdot\text{odd}$, $\iota^*:H^3(SU(n);\Ztwo)\to H^3(SU(2);\Ztwo)$ is an isomorphism.  And the map $j$ induces an isomorphism on $H_1(\quad;\Ztwo)$ 
 and hence on $H^1(\quad;\Ztwo)$. Recall that for $q=1$, $2$, and $3$, $H^q(PSU(n);\Ztwo) = \Ztwo$ generated by $x^q$, and that $H^4(PSU(n);\Ztwo)=0$,
 where $x=x_1$ (in the notation of Theorem \ref{thm:BB})
is the generator in degree $1$. Thus $j^*$ is also an isomorphism
on $H^q(\quad;\Ztwo)$ for $q=2$, $3$, and $4$.

Therefore it suffices to consider the covering $SU(2)\to PSU(2)=SO(3)$, or equivalently, $\pi:S^3\to \RP^3$.

The strategy is to show $\tilde\phi_*\neq0$ on $H_3(\quad;\Ztwo)$.
(Until further notice, all homology groups are with coefficients
in $\Ztwo$.)

Now, $\phi:\RP^3\times \RP^3\longrightarrow \RP^3$ restricted to the wedge product
$\RP^3\vee\RP^3$ is null homotopic. This induces a map from the smash product
$\varphi:\RP^3\wedge\RP^3\to \RP^3$. And since $\RP^3\wedge\RP^3$
is simply connected, $\varphi$ lifts to $S^3$. Denote this lift by
$\tilde\varphi:\RP^3\wedge \RP^3 \to S^3$. 

Since $\RP^3\times \RP^3 \to \RP^3\wedge\RP^3$ is a surjection on
homology, it suffices to show $\tilde\varphi_*$ is non-zero on
$H_3(\quad)$. As this is a question about homology in degree $3$
it is enough to consider the map $\tilde\varphi$
restricted to (any subspace containing) the $3$-skeleton of
$\RP^3\wedge\RP^3$. In particular, it suffices to show that $\tilde\varphi$ is
non-trivial on $H_3(\quad)$ when restricted to $\RP^2\wedge \RP^2\subset
\RP^3\wedge \RP^3$. The restriction $\tilde\varphi|_{\RP^2\wedge \RP^2}$ will also be denoted $\tilde\varphi$ in order to simplify notation.

The homology of $\RP^2\wedge\RP^2$ is the following (with basis
listed in $\langle \quad \rangle$):

\[ H_q(\RP^2\wedge\RP^2)=\left\{ \begin{array}{cl}
0                                           & \text{if } q=1 \\
 \langle a\otimes a \rangle =\Ztwo                        & \text{if } q=2 \\
 \langle a\otimes b, b\otimes a \rangle =\Ztwo\oplus\Ztwo  &\text{if } q=3 \\
\langle b\otimes b\rangle=\Ztwo                          & \text{if } q=4 \\
\end{array} \right.  \]
where $a\in H_1(\RP^2)$ and $b\in H_2(\RP^2)$ are generators, with
Bockstein $\beta(b)=a$. Thus $\beta(b\otimes b)=a\otimes b
+b\otimes a$ and $\beta(a\otimes b)=a\otimes a =\beta(b\otimes
a)$. For reasons that will be clear in a moment, pick a new basis
for $H_3(\RP^2\wedge\RP^2)=\langle a\otimes b, a\otimes b + b\otimes a \rangle$.

An analysis of the cell structure of $\RP^2\wedge \RP^2$ shows
that there is a cofibration sequence $\Sigma\RP^2\toby{i}
\RP^2\wedge \RP^2 \toby{q} \Sigma^2\RP^2$, where $\Sigma X$ denotes the suspension of $X$. From the long exact sequence in homology, observe that:

\begin{enumerate}

\item[(i)] the basis element $a\otimes a$ `comes from' the bottom
cell of $\Sigma\RP^2$. i.e. $i_*(\sigma a) = a\otimes a$

\item[(ii)] the basis element $a\otimes b$ `comes from' the top
cell of $\Sigma \RP^2$. i.e. $i_*(\sigma b) = a\otimes b$

\item[(iii)] the basis element $a\otimes b + b\otimes a$ `comes
from' the bottom cell of $\Sigma^2\RP^2$. i.e. $q_*(a\otimes b+
b\otimes a)= \sigma^2 a$

\item[(iv)] the basis element $b\otimes b$ `comes from' the top
cell of $\Sigma^2\RP^2$. i.e. $q_*(b\otimes b)=\sigma^2 b$
\end{enumerate}
where $\sigma:H^*(X)\to H^{*+1}(\Sigma X)$ denotes the suspension isomorphism.

Since $\tilde\varphi_*(a\otimes b +b\otimes
a)=\tilde\varphi_*(\beta (b\otimes
b))=\beta\tilde\varphi_*(b\otimes b)=0$,
 it will suffice to check $\tilde\varphi_*(a\otimes b) =z_3$.
(Using the Bockstein spectral sequence, one finds that this
corresponds to the dual statement that $\tilde\varphi^*(z_3)$ generates
the $\Ztwo$ summand, which is consistent with the previous case.
Therefore checking the above equality is precisely what is
required to prove.)
And by (ii), it suffices to show that
$\tilde\varphi_*\neq 0$ when restricted to $\Sigma\RP^2$.
By the following two facts, this is equivalent to showing that the
map $\tilde\varphi\circ i:\Sigma\RP^2\to S^3$ is essential.

\begin{fact}
The set of homotopy classes of maps $[\Sigma\RP^2,S^3]=\Ztwo$.
\end{fact}
\begin{proof} From the cofibration sequence
\[ S^2\to\Sigma\RP^2 \to S^3 \to S^3 \]
 there is a long exact sequence
\[ [S^3,S^3] \toby{-\times 2} [S^3,S^3] \to [\Sigma\RP^2,S^3] \to [S^2,S^3] \]
which gives the result. \end{proof}

\begin{fact}
The essential maps $\Sigma\RP^2\to S^3$ induce isomorphisms on
$H_3$.
\end{fact}
\begin{proof} Again, consider the cofibration sequence
$S^2\to\Sigma\RP^2 \to S^3$ and the corresponding long exact
sequence in homology. \end{proof}

Actually, it is enough to show that the composition $\Sigma\RP^2
\to S^3\toby{\pi} \RP^3$ is essential. Indeed, $\Sigma\RP^2$ is
simply connected, so $[\Sigma\RP^2,\RP^3]=[\Sigma\RP^2,S^3]$.

\noindent\textbf{Claim:} There is a homotopy commutative diagram:
\[ \begin{array}{ccccccc}
\Sigma\RP^2 & \toby{i} & \RP^2\wedge\RP^2 & \to & \RP^3\wedge\RP^3
\\
  &   & \downarrow \langle\alpha, \alpha\rangle &   & \downarrow \varphi \\
  &   & \Omega\Sigma\RP^2 & \longrightarrow & \RP^3 \\
\end{array} \]

\begin{proof}[Proof of Claim] Let $\alpha:\RP^2\to
\Omega\Sigma\RP^2$ denote the adjoint of the identity map on
$\Sigma\RP^2$. The map labeled $\langle\alpha, \alpha\rangle$ is the
Samelson Product of $\alpha$ with itself.
(Recall that the Samelson product of two maps $f,g:X\to G$ where
$G$ is an H-group --- a topological group up to homotopy --- is defined as follows. The restriction to $X\vee X$ of the
composition $[f,g]:X\times X \to G$, given by
$[f,g](x,y)=f(x)g(y)f(x)^{-1}g(y)^{-1}$, is null
homotopic, and yields a map  $\langle f,  g \rangle:X\wedge X\to
G$.)

The map $\Omega\Sigma\RP^2\to \RP^3$ is described next. (Recall that ``suspension'' is the adjoint functor to ``based loop'': $[X,\Omega Y]\cong [\Sigma X, Y]$.) The adjoint
of the inclusion $\RP^2\hookrightarrow \RP^3 \approx \Omega B\RP^3$ gives a map
$\rho:\Sigma\RP^2\to B\RP^3$, and hence a map
$\Omega(\rho):\Omega\Sigma\RP^2\to \Omega B\RP^3\approx \RP^3$. (Here, $BH$ denotes the classifying space of the topological group $H$, and $\RP^3$ is viewed as a topological group via the identification $\RP^3 = SO(3)$.)
An important point is that this map
is an H-map --- a homomorphism up to homotopy --- so that (up to homotopy) commutators map to commutators.  Therefore
the diagram homotopy commutes, and the claim is proved. 
\end{proof}

Now suppose for contradiction that the composition $\Sigma\RP^2\to \RP^2\wedge \RP^2 \to
\Omega\Sigma \RP^2 \toby{\Omega(\rho)} \RP^3$ is null homotopic. Then there exists
a lift $g:\Sigma\RP^2 \to \Omega Z$ where $Z$ is the homotopy
fiber of the map $\rho:\Sigma \RP^2 \to B\RP^3$ (and hence $\Omega Z$ is the homotopy fiber of the map $\Omega (\rho)$).
Consider the restriction of $g$ to the $2$-skeleton of
$\Sigma\RP^2$, namely $S^2 \to\Sigma \RP^2 \to \Omega Z$.  Since
$\pi_2(\Sigma RP^2)=\Ztwo$ is  torsion, then according to the following
fact, the map $S^2 \to\Sigma \RP^2 \to \Omega Z$ is null homotopic.

\begin{fact} (Proposition 6.6 in \cite{W}) $\pi_2(\Omega Z) \cong \Z$
\end{fact}

But then the composition $h:S^2\to \Sigma\RP^2 \to \Omega Z \to \Omega \Sigma
\RP^2$ is null homotopic, which  is a contradiction by the following
fact.

\begin{fact} (Proposition 6.5 in \cite{W}) $\pi_3(\Sigma\RP^2) = \Z_4$ is generated by the
composition $S^3\toby{\eta}S^2\to \Sigma\RP^2$, where $\eta$ is
the Hopf map.
\end{fact}

It is a contradiction because the map $h$ is actually the adjoint of twice the generator of
$\pi_3(\Sigma\RP^2)$. To see this, recall that $h$ is the composition of $S^2 \to \Sigma\RP^2 \to \RP^2\wedge \RP^2$, which is the inclusion of the 2-skeleton, and $\RP^2 \wedge \RP^2  \to \Omega\Sigma \RP^2$, which is the Samelson product $\langle \alpha, \alpha \rangle$. Therefore $h$ is either composition in the diagram below:
\[
\begin{array}{ccccccc}
S^2 & = & S^1\wedge S^1 & \toby{j\wedge j} & \RP^2\wedge\RP^2 \\
 & & \quad\downarrow \langle\iota, \iota\rangle &   & \downarrow \langle\alpha, \alpha\rangle \\
 & & \Omega S^2 \quad &\toby{\Omega\Sigma j}   & \Omega\Sigma\RP^2 \\
\end{array} \]
where $j:S^1\to\RP^2$ is the inclusion of the 1-skeleton, and
$\iota$ is the adjoint of the identity on $S^2$. (The square commutes, since  $\Omega\Sigma j\circ \iota= \alpha\circ j$, and Samelson product is functorial.)
Now, the Samelson product $\langle\iota,\iota\rangle$ is the
adjoint of the map $[\iota,\iota]$, called the Whitehead product.
And the Whitehead product is known to be homotopic to $2\eta$ where $\eta$ is the
Hopf map. Therefore the composition $h$ is indeed
the adjoint of twice the generator of $\pi_3(\Sigma\RP^2)$. This completes the proof of the case $G=\tilde{G}/\Zn$.
\\

The case $G=\tilde{G}/\Z_l$ where $1<l<n$  will now be considered.  It will be shown that $l_0$ is the order of $q=\frac{n}{l} \mod l$ in $\Z_l$.

 Let $P\tilde{G}$ denote $\tilde{G}/\Zn$, and let $p$ be a prime that divides $l$. Write $l=p^sl'$ and
$n=p^rn'$ where $p$ is relatively prime to both $l'$ and $n'$.
Consider the covering homomorphism $f:G\to P\tilde{G}$ with kernel $\Z_n/\Z_l = \Z_{n/l}$. From the diagram:
\[ \begin{array}{ccc}
\tilde{G}  &  =& \tilde{G} \\
\downarrow   &        & \downarrow  \\
G & \toby{f} & P\tilde{G} \\
\end{array} \]
 it follows that $\tilde{\phi}_G=\tilde{\phi}_{P\tilde{G}}\circ(f\times f)$.
If $s=r$, then \cite{BB} show that $f$
induces an isomorphism on $H^*(\quad;\Zp)$. Therefore, for such
primes $p$, $\tilde\phi^*_G(z_3)$ is the reduction mod $p$ of the
generator of the torsion summand of order $l$ in $H^3(G\times G;\Z)$.

If $s<r$, then by a method similar to the previous cases $\tilde\phi_{P\tilde{G}}^*(z_3)=x_1\otimes y - y \otimes x_1\in H^3(P\tilde{G}\times P\tilde{G};\Z_{p^s})$ is the reduction mod $p^s$ of generator of a $\Z_{p^r}$ summand in $H^3(P\tilde{G}\times P\tilde{G};\Z)$.  Now, the induced homomorphism $f_*:H_1(G;\Z) \to H_1(P\tilde{G};\Z)$ sends a generator in $H_1(G;\Z)\cong\Z_l$ to $q=\frac{n}{l}$ times a generator of $H_1(P\tilde{G};\Z)\cong \Z_n$. Therefore, $f^*:H^1(P\tilde{G};\Z_{p^s}) \to H^1(G;\Z_{p^s})$ sends a generator in $H^1(P\tilde{G};\Z_{p^s}) \cong \Z_{p^s}$ to $p^{r-s}$ times a generator in $H^1(G;\Z_{p^s})\cong\Z_{p^s}$. Also, from \cite{BB}, $f^*$ is an isomorphism in dimension 2, which gives $(f\times f)^*(x_1\otimes y - y\otimes x_1)=p^{r-s}(x_1\otimes y - y\otimes x_1)$, $p^{r-s}$ times the reduction mod $p^s$  of a generator of a $\Z_{p^s}$ summand in $H^3(G\times G;\Z)$.   

Therefore, for primes $p$ with $s=r$, the reduction mod $p^s$ of the obstruction $\tilde\phi_G^*(z_3)\in \Z_l\subset H^3(G\times G;\Z)$ generates a $\Z_{p^s}$ summand, and for primes $p$ with $s<r$, the reduction mod $p^s$ of the obstruction $\tilde\phi_G^*(z_3)$ is $p^{r-s}$ times a generator of a $\Z_{p^s}$ summand.
These observations imply that $l_0$, the order of the obstruction $\tilde\phi_G^*(z_3) \in \Z_l$ is the order of $q=\frac{n}{l} \mod \Z_l$ in $\Z_l$.
 This completes the proof of Theorem \ref{thm:An}. \end{proof}


\subsection{The $C_n$'s}

\begin{thm} Let $\tilde{G}=Sp(n)$, with center $\Ztwo$, and let $G=\tilde{G}/\Ztwo$. If $n$ is even, then $l_0=1$. And if $n$ is odd, $l_0=2$.
\end{thm}
\begin{proof} First suppose $n$ is even. From the map
$g:Sp(n)\to SU(2n)$, which is known to induce an epimorphism on
cohomology (see \cite{BB}), there is a diagram:
\[
\begin{array}{ccc}
     Sp(n)     &\toby{g} & SU(2n)  \\
\downarrow  &     & \downarrow   \\
 PSp(n)     &\toby{\bar{g}} & SU(2n)/\Ztwo
\end{array}\]
\noindent which shows $\tilde\phi_{Sp(n)}^*(z'_3)=\tilde\phi_{Sp(n)}^*(g^*(z_3))=
(\bar{g}\times\bar{g})^*(\tilde\phi^*_{SU(n)}(z_3))=0$, where $z'_3=g^*(z_3)$ denotes  the generator of $H^3(Sp(n))$.

Before proceeding with the remaining case, recall Baum and
Browder's result on the cohomology of $PSp(n)$~\cite{BB}:

\begin{thm} Write $n=2^r n'$ where $n'$ is odd. As an algebra:
\[H^*(PSp(n);\Ztwo) = \Lambda(b_3,\ldots,\hat{b}_{2^{r+2}-1}, \ldots , b_{4n-1})
    \otimes \Ztwo[v]/(v^{2^{r+2}} )\]
\end{thm}

Now suppose $n$ is odd. The diagonal inclusion $\iota:Sp(1)\to Sp(n)$ induces a
map $j:PSp(1)\to PSp(n)$.  And as in the proof of Theorem \ref{thm:An}, $\iota$ and $j$ induce isomorphisms on $H^q(\quad;\Ztwo)$ in dimensions $q\leq 3$. But $Sp(1)=SU(2)$, and this was done in the last section. Recall that in this case,
$\tilde\phi^*(z_3)$ was the reduction mod $2$ of the generator of
the torsion summand. As $Sp(n)$ has only $2$-torsion, the proof is complete.
\end{proof}


\subsection{The $B_n$'s and the $D_n$'s}

Recall that the $B_n$'s have simply connected representative
$\tilde{G}=Spin(2n+1)$ and center $\Ztwo$, where of course
$PSpin(2n+1)=SO(2n+1)$.
And the $D_n$'s have simply connected representative
$\tilde{G}=Spin(2n)$ whose center is $\Z_4$ if $n$ is odd, and
$\Ztwo\oplus\Ztwo$ if $n$ is even. In this case, the non-simply
connected representatives occur in three types
$PSpin(2n)=PSO(2n)$, $SO(2n)$, and the semi-spinor group $Ss(2n)$, which are
$\tilde{G}/\Gamma$ for $\Gamma$ the full center, and different
central subgroups of order $2$. In particular, (see \cite{IKT}) if $a$ denotes the generator of the kernel of the double cover $Spin(2n)\to SO(2n)$, and $b$ denotes another generator of the center $\Ztwo \oplus \Ztwo$ of $Spin(2n)$, then $Ss(2n) = Spin(2n)/\langle b \rangle$. For later use (in the proof of Theorem \ref{lateruse}), note that there is a covering homomorphism $f:Ss(2n) \to PSO(2n)$.


\subsubsection{$G=SO(n)$, $n\geq 7$}

Consider first $G=SO(n)$, with universal covering group
$\tilde{G}=Spin(n)$. The covering projection $\pi:Spin(n)\to
SO(n)$ induces an epimorphism on $H^3(\quad;\Ztwo)$, and according
to ~\cite{MT}, the generator of $H^3(Spin(n);\Ztwo)$ is $\pi^*(x_3)$, where $x_3\in H^3(SO(n);\Ztwo)$ is primitive.
It therefore suffices to calculate $\phi^*(x_3)$.  But since $x_3$ is
primitive,  $\phi^*(x_3)=0$. This proves:

\begin{thm}For $G=SO(n)$, $l_0=1$.
\end{thm}

\begin{remark}
\begin{enumerate} \item[(a)] This completes the analysis of the
$B_n$'s.

\item[(b)] This completes the analysis of the $SO(2n)$ type within
the $D_n$'s.
\end{enumerate}
\end{remark}


\subsubsection{$G=PSO(2n)$, $n\geq 4$}

\begin{thm}Let $G=PSO(2n)$. If $n$ is odd $l_0=4$. If $n$ is even
$l_0=2$. \label{lateruse}
\end{thm}

\begin{proof} The method is very much like the proof of \ref{thm:An}, relying on the results of Baum and Browder~\cite{BB}
concerning the cohomology of $PSO(2n)$. That is, one finds that
the covering projection $Spin(2n)\to PSO(2n)$ induces an
epimorphism on $H^3(\quad;\Ztwo)$ (when $n\neq 2$ mod $4$). Then one may compute
$\phi^*$ instead, and use the knowledge of the Bockstein spectral
sequence to sort out what happens integrally.

When $n=2$ mod $4$, this method fails. Write $n=4k$, where $k$ is odd. Since $H^3(PSO(4k)^2;\Z)$ contains only 2-torsion, it will suffice to show that $\tilde\phi^*\neq 0$ on $H^3(\quad;\Z)$. To that end, consider the diagram
\[ \begin{array}{ccc}
Spin(4k) & = & Spin(4k) \\
\downarrow & & \downarrow \\
Ss(4k) & \toby{f} & PSO(4k)
\end{array} \]
Therefore, $(f\times f)^*\tilde\phi_{PSO(4k)}^*=\tilde\phi^*_{Ss(4k)}$. Since $\tilde\phi^*_{Ss(4k)}\neq 0$ (by Theorem \ref{mapcon}), the map $\tilde\phi^*_{PSO(4k)}$ is non-zero. \end{proof}


\subsubsection{$G=Ss(4n)$}

\begin{thm} Let $G=Ss(4n)$. If $n$ is odd, $l_0=2$. If $n$ is
even, $l_0=1$. \label{mapcon}
\end{thm}

\begin{proof} If $n$ is even,  the covering projection
$Spin(4n)\to Ss(4n)$ induces an epimorphism on $H^3(\quad;\Ztwo)$ (see \cite{IKT}), and since the generator in $H^3(Ss(4n);\Ztwo)$ is primitive,
$\phi^*=0$.

If $n$ is odd,  there is a map $g:\RP^3 \to Ss(4n)$ (constructed in the next paragraph) that induces a weak homotopy equivalence in degrees  $\leq 3$.  Therefore $g$ is covered by a map $\tilde{g}:S^3 \to Spin(4n)$ that induces an isomorphism on $H^3(\quad;\Z)$. And since $\mu_{\RP^3}^*g^* = (g\times g)^*\mu_{Ss(4n)}^*$ on $H^3(\quad;\Z)$, it follows that $\tilde\phi_{\RP^3}^*\tilde{g}^*=(g\times g)^*\tilde\phi_{Ss(4n)}^*$. Therefore this case reduces to Theorem \ref{thm:An}, and the image of $\tilde\phi^*$ is $2$-torsion.

The map $g$ may be constructed as follows. Recall that $\RP^3$ is obtained (as a CW complex) from $\RP^2$ by attaching a 3-cell with attaching map $S^2\to \RP^2$ the quotient map. Similarly, $\RP^2$ is obtained from $S^1$ by attaching a 2-cell with attaching map $2:S^1\to S^1$, the degree $2$ map.  Let $f:S^1 \to Ss(4n)$ represent a generator of $\pi_1(Ss(4n))=\Ztwo$. The composition $f\circ 2$  is null homotopic, and therefore extends to a map $F: \RP^2 \to Ss(4n)$. And since the composition $S^2 \to \RP^2 \to Ss(4n)$ is null homotopic  (as $\pi_2(Ss(4n))=0$), $F$ extends to a map $g:\RP^3 \to Ss(4n)$. By the structure of the cohomology rings  of $\RP^3$ and $Ss(4n)$ (see \cite{IKT}),  $g$ induces an isomorphism in $H^q(\quad;\Z)$ for $q\leq3$, since $g$ induces an isomorphism on $\pi_1(\quad)$. \end{proof}


\subsection{The Exceptionals}

Of the exceptional groups, only $E_6$ and $E_7$ have non-zero
centers: namely $\Z_3$ and $\Ztwo$ respectively. As usual, let
$PG$ denote the quotient of $G$ by its center.

\begin{thm}For $G=PE_6$, $l_0=3$. And for $G=PE_7$, $l_0=2$.
\end{thm}

\begin{proof} Kono \cite{Ko} provides the necessary information about the cohomology of $PE_6$ and $PE_7$. In
particular, the covering projection $E_6\to PE_6$ is onto on
$H^3$, so it suffices to compute $\phi^*$ instead. The
calculation is similar to the ones of the previous section, and one finds again that $\phi^*(x_3)$ is the reduction mod $3$ of the
torsion summand.

For $PE_7$, there is a map $g:\RP^3\to PE_7$ which is a weak homotopy 
equivalence dimensions $\leq 3$, as in the proof of Theorem \ref{mapcon}.  \end{proof}


\appendix


\section{Pre-quantization of Hamiltonian loop group manifolds} \label{sec:loopgroup}

For a compact Lie group $G$, \cite{AMM} shows that there is a bijective
correspondence between quasi-Hamiltonian spaces with group-valued
moment map and Hamiltonian loop group spaces. This correspondence is discussed next, in relation to pre-quantization when $G$ is simple and simply connected.

It may be useful to quickly review some background concerning loop groups, and their central extensions.  Technical issues concerning smoothness conditions on  mapping spaces are ignored; the reader may consult  \cite{MW}, where these details are considered.

Let $LG=\mathrm{Map}(S^1,G)$, and let $\Lg =
\mathrm{Map}(S^1,\mathfrak{g})$ denote its Lie algebra. Let
$(\cdot,\cdot)$ be an invariant inner product on $\mathfrak{g}$.

Define $\Lg^*:=\Omega^1(S^1;\mathfrak{g})$ with pairing $\Lg\times
\Lg^*\to \R$ given by $(\xi,A)\mapsto \int_{S^1}(\xi,A)$, which
identifies $\Lg^*\subset (\Lg)^*$.

  Let $\Lghat=\Lg\oplus \R$ with Lie bracket determined by:
\[ [(\xi_1,t_1),(\xi_2,t_2)]:= ([\xi_1,\xi_2], \int_{S^1}(\xi_1,d\xi_2))
\]

Dually, define $\Lghat^*:=\Lg^*\oplus\R$ and consider the pairing
$\Lghat\times \Lghat^*\to\R$ given by $((\xi,a), (A,\lambda)) = \int_{S^1} (\xi,A)+ a\lambda$ .

The following result is well known (see \cite{PS}).
\begin{thm} There is a unique (non-trivial) central extension $\LGhat$ of
$LG$ by $S^1$ with Lie algebra $\Lghat$ if and only if the inner product $(\cdot,\cdot)$ is $kB$, where $k \in \N$, and $B$ denotes the basic inner product. The coadjoint action of
$\LGhat$ on $\Lghat^*$ factors through $LG$, and is explicitly
given by:
\[g\cdot(A,\lambda)=(\mathrm{Ad}_g(A)-\lambda g^*\theta^R,\lambda)   \]
\end{thm}

Identifying $\Lg^*$ with $\Lg^* \times \{ \lambda \} \subset \Lghat^*$ gives an action of $LG$ on $\Lg^*$, which is called the level $\lambda$ action.  When $\lambda=1$, this action coincides with the usual action of $LG$ (viewed as the group of gauge transformations) on the space of connections on the trivial principal $G$ bundle over $S^1$.

From now on, $\LGhat$ will denote the basic central extension (i.e. corresponding to $k=1$). Note that this central extension is simply connected.

\begin{defn} A Hamiltonian loop group manifold at level $\lambda \in \R$ is a triple
$(\mathcal{M},\sigma,\varphi)$ consisting of a symplectic Banach manifold
$(\mathcal{M}, \sigma)$ equipped with a Hamiltonian $LG$ action, with equivariant moment map $\varphi:\mathcal{M} \to \Lg^*$  satisfying the moment map condition
$\iota_{\xi^\sharp}\sigma = d(\varphi,\xi)$.  (Note that $\Lg^*$ comes equipped with the level $\lambda$ action of $LG$.)
\end{defn}

The $LG$-space $\mathcal{M}$ is often viewed as an $\LGhat$ space
where the central circle acts trivially, with constant moment map $\lambda$.

\begin{defn} A pre-quantization of $(\mathcal{M},\sigma,\varphi)$, a Hamiltonian $LG$-space at level $\lambda$, is an $\LGhat$-equivariant
principal $U(1)$-bundle $\pi:P\to \mathcal{M}$ equipped with an $\LGhat$-invariant connection $A$  satisfying  $\pi^*\sigma= -dA $, and $(\pi^*\varphi,\zeta)= A(\zeta^\sharp)$, where $\zeta^\sharp$ is the generating vector field on $P$ for $\zeta\in\Lghat$.
\end{defn}

\begin{remark} \label{remark:obs}
A pre-quantization of a Hamiltonian $LG$-space places restrictions on both the level $\lambda$ (which must be an integer)  and the symplectic form $\sigma$ (which must be integral). Indeed, since the moment map for the (trivial) circle action on $\mathcal{M}$ is the constant  $\lambda$, one sees that corresponding circle action on the fibers of  $\pi:P\to \mathcal{M}$  is of weight $\lambda$, and must therefore be an integer.
\end{remark}

Since the based loop group $\Omega G\subset LG$ acts freely on
$\Lg^*$, by equivariance of $\varphi$, it acts freely on
$\mathcal{M}$; therefore there is a pull-back diagram:
\[ \begin{array}{ccl}
\mathcal{M} & \toby{\varphi} & \Lg^* \\
q \downarrow       &              &  \downarrow \Hol \\
M           & \toby{\phi}    & G \\
\end{array} \]
When $\varphi$ is proper, the quotient $M=\mathcal{M}/\Omega G$ is
a finite dimensional manifold.  The following result of \cite{AMM} describes the
correspondence between quasi-Hamiltonian spaces with group-valued
moment map and Hamiltonian loop group spaces.

\begin{thm}(Equivalence Theorem) Let
$(\mathcal{M},\sigma,\varphi)$ be a Hamiltonian $LG$-space with
proper moment map $\varphi$, and let $M$ and $\phi$ be as above.
There is a unique two-form $\omega$ on $M$ with $q^*\omega =
\sigma + \varphi^*\varpi$ with the property that $(M,\omega,\phi)$
is a quasi-Hamiltonian $G$-space. Conversely, given a quasi-Hamiltonian $G$-space
$(M,\omega,\phi)$  there is a
unique Hamiltonian $LG$-space $(\mathcal{M},\sigma,\varphi)$ such
that $M=\mathcal{M}/\Omega G$.
\end{thm}

Before relating the pre-quantization of Hamiltonian loop group spaces with quasi-Hamiltonian $G$-spaces, consider the following situation. Suppose that $(\mathcal{M}, \sigma, \varphi)$ is a level $k$ ($k\in\N$) Hamiltonian $LG$-space, and suppose $\pi:P\to \mathcal{M}$ is a principal $U(1)$ bundle with connection $A$, whose curvature is $\sigma$.  The formula
\[ \zeta^\sharp = \mathrm{Lift}(\zeta^\sharp) + 
				(\pi^*\varphi,\zeta)\frac{\partial}{\partial \theta}	 \]
due to Kostant \cite{Kos} gives a lift of the $\Lghat$-action on $\mathcal{M}$ to an $\Lghat$-action on $P$, which preserves $A$. Here, $\mathrm{Lift}(X)$ denotes the horizontal lift of the vector field $X$ on $\mathcal{M}$ determined by the connection $A$, $\zeta^\sharp$ denotes the generating vector field of $\zeta\in \Lghat$ (on $\mathcal{M}$ or $P$, depending on the context), and $\frac{\partial}{\partial \theta}$ denotes the generating vector field of the $U(1)$-action on $P$.

As $\LGhat$ is connected and simply connected, the $\Lghat$-action on $P$ integrates to an $\LGhat$-action, yielding an $\LGhat$-equivariant principal $U(1)$-bundle satisfying the conditions of a pre-quantization.  In particular, the obstructions in remark \ref{remark:obs} are the only ones (compare with Proposition \ref{prop:eq-pqispq}):

\begin{prop} Suppose that $(\mathcal{M},\sigma,\varphi)$ is a level $k$ Hamiltonian $LG$-space ($k\in\N$).  If the cohomology class $[\sigma]\in H^2(\mathcal{M};\R)$ is integral, then $(\mathcal{M},\sigma,\varphi)$ admits a pre-quantization. 
\end{prop}

The above proposition and the theorem below are based on the unpublished work of A. Alekseev, and E. Meinrenken \cite{AM}.

\begin{thm} \cite {AM} Under the correspondence in the Equivalence Theorem, a pre-quantization of the quasi-Hamiltonian $G$-space $(M,\omega,\phi)$ corresponds to a
pre-quantization of the Hamiltonian
$LG$-space $(\mathcal{M},\sigma,\varphi)$. 
\end{thm}
\begin{proof} By the previous proposition it suffices to show that $[(\omega,\eta)]$ is integral if and only if $[\sigma]$ is integral.
The above pull-back diagram induces a map $(q,\Hol)^*:H^3(\phi)\to H^3(\varphi)$, which  is an isomorphism for connectivity reasons. Indeed, the above pull-back diagram identifies $\mathcal{M}$ as the homotopy fiber of the map $\phi:M\to G$. The map of pairs $(q,\Hol)$ is homotopy equivalence in dimensions $\leq q$ if and only if the map of pairs $(\varphi,\phi)$ is a homotopy equivalence in dimensions $\leq q$.  By Ganea's Theorem (see \cite{S}) $(\varphi,\phi)$ is a homotopy equivalence in dimensions $\leq 3$, as $G$ is $2$-connected and $\mathcal{M}$ is connected.
Alternatively, one could argue with the Serre exact sequence for the homotopy fibration $\mathcal{M}\to M\to G$.

Now, the connecting map $\delta:H^2(\mathcal{M})\to H^3(\varphi)$ in the long exact sequence in cohomology for the map $\varphi$ is an isomorphism, induced from the chain map $\alpha\mapsto (\alpha,0)$. Thus the theorem is proved after observing:
\begin{align*}
 (q,\Hol)^*[(\omega,\eta)] & = [(q^*\omega,\Hol^*\eta)]\\
		& = [(\sigma +\varphi^*\varpi,d\varpi)] \\
		& = [(\sigma,0)+d(0,\varpi)] \\
		& = [(\sigma,0)]  \quad  \\
\end{align*}
\end{proof}


\end{document}